\documentclass[12pt]{amsart}

\usepackage{psfrag}
\usepackage{color}
\usepackage{tikz}
\usetikzlibrary{matrix,arrows}
\usepackage{graphicx,graphics}
\usepackage{fullpage,amssymb,amsfonts,amsmath,amstext,amsthm,amscd,verbatim,enumerate}
\usepackage[T1]{fontenc}

\begin{document}

\newtheorem{theorem}{Theorem}[section]
\newtheorem{result}[theorem]{Result}
\newtheorem{fact}[theorem]{Fact}
\newtheorem{example}[theorem]{Example}
\newtheorem{conjecture}[theorem]{Conjecture}
\newtheorem{lemma}[theorem]{Lemma}
\newtheorem{proposition}[theorem]{Proposition}
\newtheorem{corollary}[theorem]{Corollary}
\newtheorem{facts}[theorem]{Facts}
\newtheorem{props}[theorem]{Properties}
\newtheorem*{thmA}{Theorem A}
\newtheorem{ex}[theorem]{Example}
\theoremstyle{definition}
\newtheorem{definition}[theorem]{Definition}
\newtheorem{remark}[theorem]{Remark}
\newtheorem*{defna}{Definition}

\newcommand{\notes} {\noindent \textbf{Notes.  }}
\newcommand{\note} {\noindent \textbf{Note.  }}
\newcommand{\defn} {\noindent \textbf{Definition.  }}
\newcommand{\defns} {\noindent \textbf{Definitions.  }}
\newcommand{\x}{{\bf x}}
\newcommand{\z}{{\bf z}}
\newcommand{\B}{{\bf b}}
\newcommand{\V}{{\bf v}}
\newcommand{\T}{\mathbb{T}}
\newcommand{\Z}{\mathbb{Z}}
\newcommand{\Hp}{\mathbb{H}}
\newcommand{\D}{\mathbb{D}}
\newcommand{\R}{\mathbb{R}}
\newcommand{\N}{\mathbb{N}}
\renewcommand{\B}{\mathbb{B}}
\newcommand{\C}{\mathbb{C}}
\newcommand{\ft}{\widetilde{f}}
\newcommand{\dt}{{\mathrm{det }\;}}
 \newcommand{\adj}{{\mathrm{adj}\;}}
 \newcommand{\0}{{\bf O}}
 \newcommand{\av}{\arrowvert}
 \newcommand{\zbar}{\overline{z}}
 \newcommand{\xbar}{\overline{X}}
 \newcommand{\htt}{\widetilde{h}}
\newcommand{\ty}{\mathcal{T}}
\renewcommand\Re{\operatorname{Re}}
\renewcommand\Im{\operatorname{Im}}
\newcommand{\tr}{\operatorname{Tr}}

\newcommand{\ds}{\displaystyle}
\numberwithin{equation}{section}

\renewcommand{\theenumi}{(\roman{enumi})}
\renewcommand{\labelenumi}{\theenumi}

\title{On quasiregular linearizers}

\author{Alastair Fletcher}
\address{Department of Mathematical Sciences, Northern Illinois University, DeKalb, IL 60115-2888. USA}
\email{fletcher@math.niu.edu}

\author{Douglas Macclure}
\address{Department of Mathematical Sciences, Northern Illinois University, DeKalb, IL 60115-2888. USA}
\email{doug.macclure@gmail.com}

\begin{abstract}
Linearization is a well-known concept in complex dynamics. If $p$ is a polynomial and $z_0$ is a repelling fixed point, then there is an entire function $L$ which conjugates $p$ to the linear map $z\mapsto p'(z_0)z$. This notion of linearization carries over into the quasiregular setting, in the context of repelling fixed points of uniformly quasiregular mappings. In this article, we investigate how linearizers arising from the same uqr mapping and the same repelling fixed point are related. In particular, any linearizer arising from a uqr solution to a Schr\"oder equation is shown to be automorphic with respect to some quasiconformal group.
\end{abstract}

\maketitle

\section{Introduction}

\subsection{Complex dynamics}
Complex dynamics is a currently very active field of research. The groundwork for this exciting field of analysis was laid in the 19th and early 20th centuries by Poincar\'{e}, Julia,  Fatou, and several other mathematicians. Only sporadic progress was seen until the early 1980s, when computers provided a fascinating glimpse at the austere beauty and symmetry that hides within complex dynamics. Douday, Hubbard and Sullivan were the first to introduce quasiconformal methods into complex dynamics which helped bring about new impetus. See \cite{BF} for quasiconformal methods in complex dynamics.

A central theme in this field of mathematics is linearization, which allows one to study local behavior near certain fixed points by  conjugating to a simpler and more easily understood function. In complex dynamics, Koenig's Linearization Theorem illustrates this idea. Given a polynomial $p:\C \to \C$ with fixed point $z_0$ and multiplier $p'(z_0)$ with $|p'(z_0)| \notin \{ 0,1\}$, there exists a neighborhood $V$ of $0$ and a holomorphic function $L:V\to \C$ such that $p(L(z)) = L(p'(z_0)z)$ holds in $V$. If $0<|p'(z_0)|<1$, then $z_0$ is an attracting fixed point and the functional equation $p(L(z)) = L(p'(z_0)z)$ implies that the domain of definition of $L^{-1}$ can be extended to the basin of attraction of $z_0$. On the other hand, and of more interest to this paper, if $|p'(z_0)| >1$, then $z_0$ is a repelling fixed point and the functional equation implies that $L$ can be extended to a transcendental entire function.
See \cite{Milnor} for further analysis.

\subsection{Quasiregular dynamics}

Quasiregular mappings can be seen as a higher dimensional analogue to holomorphic mappings of a single complex variable. Informally, these are mappings with a bounded amount of distortion. Since the generalized Liouville's Theorem states that the only holomorphic mappings in higher dimensions are M\"obius transformations, it is natural to allow distortion to have an interesting function theory.
 
This means that the iteration of quasiregular mappings is the natural higher dimensional analogue of complex dynamics.  There are results from complex analysis that hold in the setting of quasiregular mappings, for example there are quasiregular versions of Picard's Theorem and Montel's Theorem. Since the complex version of these results are crucial to the theory of complex dynamics, the fact there are quasiregular versions indicates a similar theory can be established. See \cite{Rickman} for an introduction to quasiregular mappings and \cite{Berg} for an introduction to quasiregular dynamics.

\subsection{Quasiregular linearizers}

The most direct quasiregular analogue of holomorphic functions in the plane are uniformly quasiregular mappings (uqr mappings for short), for which there is a uniform bound on the distortion of the iterates. For such mappings, the Julia set and Fatou set can be defined and they partition space, just as in the complex case.
The fixed points of uqr mappings have been classified \cite{H.M.M.} in a similar fashion to the classification of fixed points of holomorphic functions. In particular, there are repelling fixed points of uniformly quasiregular mappings. See below for the precise definition.

Since quasiregular mappings are not differentiable everywhere, it is possible that a uqr mapping may not be differentiable at a fixed point. This technical difficulty in extending the idea of linearization to the uqr setting is overcome in \cite{H.M.M.}, where a generalized derivative for uqr mappings is defined, and the existence of a quasiregular linearizer is proved. The dynamics of such linearizers were studied in \cite{Fletcher1}, and in this paper we propose a continuation of this analysis. 

It is well-known, see for example \cite{Milnor}, that if $L,M$ are both linearizers of a given polynomial at a repelling fixed point, then there exists $\lambda \in \C$ such that $L(z) = M(\lambda z)$. In this paper, we will be interested in analogous results in the quasiregular setting which generalize this observation. Further, we will show that automorphic quasiregular mappings arise as linearizers.

This paper is organised as follows. In section 2, we briefly recall some definitions that are important for this paper. In section 3, we state our main results and give discussions on them. In section 4, we show how quasiregular linearizers are related. In section 5, we specialize to the case of differentiability. In section 6, we relate automorphic quasiregular mappings to linearizers via the Schr\"oder equation. Finally in section 7, we illustrate our results with a specific example.

The authors would like to thank Deepak Naidu and Dan Nicks for helpful comments.

\section{Preliminaries}

\subsection{Quasiregular mappings}
{\it Quasiregular mappings} are Sobolev mappings in $W^1_{n,loc}(\mathbb{R}^n)$ where there is a uniform bound on distortion. More precisely, a mapping $f: E \rightarrow {\R}^n$ defined on a domain $E \subset {\R}^n$ is quasiregular if $f$ belongs to the above Sobolev space and there exists $K \in [1,\infty)$ such that 
\begin{equation} 
\label{eq:1}
|f'(x)|^n \leq KJ_f(x) 
\end{equation}
almost everywhere.  Here, $J_f(x)$ denotes the Jacobian determinant of $f$ at $x \in E$.  The smallest constant $K \geq 1$ for which \eqref{eq:1} holds is called the outer dilation $K_O(f)$.  If $f$ is quasiregular, then we also have 
\begin{equation}
\label{eq:2}
J_f(x) \leq K' \inf_{|h| = 1}|f'(x)h|^n
\end{equation}
almost everywhere in $E$ for some $K' \in [1,\infty).$  The smallest constant $K' \geq 1$ for which \eqref{eq:2} holds is called the inner dilation $K_I(f).$  The dilation $K = K(f)$ of $f$ is the larger of $K_O(f)$ and $K_I(f)$, and we then say that $f$ is $K$-quasiregular. Quasiconformal mappings are injective quasiregular mappings.

Informally, a quasiregular mapping sends infinitesimal spheres to infinitesimal ellpsoids with bounded eccentricity. 
A uniformly quasiregular mapping $f$ is one for which there exists $K\ge 1$ such that $K(f^m) \leq K$ for all $m\geq 1$. Note that the Sobolev condition just implies that a quasiregular mapping is only differentiable almost everywhere.

Suppose that $f:\R^n \to \R^n$ is quasiregular. If $\lim _{|x| \to \infty} f(x)$ exists (in $\overline{\R^n}$), then $f$ is said to be of polynomial type, whereas if this limit does not exist then $f$ is said to be of transcendental type. This is in direct analogy with polynomials and transcendental entire functions in the plane.

\subsection{Generalized derivatives and Infinitesimal spaces}\label{sec:t}

Let $f$ be a uniformly quasiregular mapping. Hinkkanen, Martin and Mayer \cite{H.M.M.} define the set $Df(x_0)$ of generalized derivatives of $f$ at $x_0$ by the set of limits of
\[ \lim _{k\to \infty} \frac{ f(x_0 + \rho_k x) - f(x_0) }{\rho_k},\]
where $\rho_k \to 0$ as $k\to \infty$. That there is a chance of such a limit existing at all, is due to the fact that uqr mappings are bi-Lipschitz in a neighborhood of a fixed point \cite{H.M.M.}.
If there is only one element in $Df(x_0)$, we will call $Df(x_0)$ simple.
If $f$ is differentiable at $x_0$, then $Df(x_0)$ just contains the linear mapping $x\mapsto f'(x_0)x$ and so $Df(x_0)$ is simple.

A fixed point $x_0$ of $f$ is called repelling if one, and hence every, element of $Df(x_0)$ is a loxodromic repelling uniformly quasiconformal map, that is, of the form $\psi :\R^n \to \R^n$ with $\psi(0) =0$, $\psi(\infty ) = \infty$ and $\psi^m(x) \to \infty$ for every $x\neq 0$.

\begin{theorem}[Theorem 6.3, \cite{H.M.M.}]
Let $f$ be uqr, $x_0$ be a repelling fixed point of $f$ and $\psi \in Df(x_0)$. Then there exists a transcendental quasiregular map $L:\R^n \to \R^n$ such that $f\circ L = L \circ \psi$.
\end{theorem}

Such mappings $L$ arise as limits of subsequences of
\[ f^k(x_0 + \rho_kx),\]
for an appropriate sequence $\rho_k \to 0$. We make the following definition.

\begin{definition}
Given a uqr map $f$ with repelling fixed point $x_0$, we define the set of Poincar\'e linearizers of $f$ at $x_0$ by
\[ \mathcal{L}(x_0,f) = \{ L:\R^n\to\R^n \text{ quasiregular } | f\circ L = L \circ \psi \text{ for some }\psi \in Df(x_0) \},\]
and we also assume $L\in \mathcal{L}(x_0,f)$ is injective in a neighborhood of $0$.
We further define 
\[ \mathcal{L}_1(x_0,f) = \{ \text{ limits of subsequences of }  f^k(x_0 + \rho_kx) \},\]
for some sequence $\rho_k \to 0$. 
\end{definition}

We clearly have $\mathcal{L}_1 \subset \mathcal{L}$.

We remark that a slightly different notion of Poincar\'e linearizers was used in \cite{Fletcher1}. There, it was observed that every generalized derivative of a repelling fixed point of a uqr mapping is quasiconformally conjugate to $T:x\mapsto 2x$, and hence linearizers were defined by quasiregular mappings $L$ such that $f\circ L = L\circ T$. Here, we will be interested in the generalized derivatives themselves, and not in conjugating them to $T$.

A more general notion of infinitesimal space for quasiregular mappings, and not just for uqr mappings, was introduced in \cite{G.M.R.V.}. If $f:\R^n \to \R^n$ is a non-constant quasiregular mapping, let $r(x_0,f,\rho)$ be the mean radius of the image of the ball of radius $\rho$ centered at $x_0$ under $f$, that is
\[ r(x_0,f,\rho) = \left ( \frac{| f(B(x_0,\rho))|}{\Omega_n} \right )^{1/n},\]
where $|E|$ denotes $n$-dimensional volume of $E$ and $\Omega_n$ is the volume of the unit ball in $\R^n$. Then the infinitesimal space is
\[ T(x_0,f) = \left \{ \text{ limits of subsequences of } \frac{f(x_0+\rho_k x)-f(x_0) }{r(x_0,f,\rho_k)} \right \},\]
as $k\to \infty$, where $\rho_k \to 0$ as $k\to \infty$. Every element of $T(x_0,f)$ is a quasiregular mapping of polynomial type. See \cite{G.M.R.V.} for more details.

If $f$ is differentiable at $x_0$ with non-degenerate derivative, then $T(x_0,f)$ just contains the normalized derivative $x \mapsto f'(x_0)/J_f(x_0)^{1/n} x$.
If $T(x_0,f)$ contains only one mapping $g$, then $g$ is called {\it simple}. By \cite[Theorem 4.1]{G.M.R.V.}, if $g$ is simple, then it is homogeneous, that is, there exists $D>0$ such that for all $r>0$, 
\begin{equation}
\label{eq:hom} g(rx) = r^Dg(x),
\end{equation}
for all $x\in \R^n$.

\section{Statement of results}

We are now in a position to state our results.

\begin{theorem}\label{thm:1}
Let $n\geq 2$ and $f:\R^n \to \R^n$ be a uqr mapping with repelling fixed point $x_0$. Suppose that $L,M \in \mathcal{L}(x_0,f)$. Then there exists a quasiconformal mapping $G:\R^n \to \R^n$ such that $M = L\circ G$.
\end{theorem}

There are some conditions we can place on $f$ to get a stronger equivalence between linearizers.

\begin{corollary}\label{cor:0}
In addition to the hypotheses of Theorem \ref{thm:1}, suppose that $L,M \in \mathcal{L}_1(x_0,f)$ arise as limits of
\[ L_j(x) =  f^{k_{u_j}}(x_0+xp_{k_{u_j}}) ,\quad M_j(x) =   f^{k_{v_j}}(x_0+xq_{k_{v_j}}),\]
where $u_j = v_j$ for infinitely many $j$. Then there exists $r>0$ such that $M(x) = L(rx)$.
\end{corollary}

\begin{corollary}\label{cor:1}
In addition to the hypotheses in Theorem \ref{thm:1}, suppose that $Df(x_0)$ is simple, i.e. $Df(x_0) = \{ \psi \}$. Then if $L,M \in \mathcal{L}(x_0,f)$, there exists a quasiconformal mapping $G$ that commutes with $\psi$ such that $M=L\circ G$.  
\end{corollary}

If $f$ is uqr and differentiable at $x_0$, then $Df(x_0)$ is simple and contains only $x \mapsto f'(x_0)x$. Now, since $f$ is uqr, this places a restriction on what $f'(x_0)$ can be. 

\begin{lemma}\label{lem:uqr}
Let $A$ be a real matrix representing a linear map $\psi : \R^n \to \R^n$. Then if $\psi$ is uniformly quasiconformal, all eigenvalues of $A$ must have the same modulus and the Jordan canonical form of $A$ (over $\C$) is diagonal.
\end{lemma}

In particular, $\psi$ is conjugate to a composition of a scaling and a rotation.
Denote by $Z(A)$ the centralizer of $A$ in the set of $n\times n$ matrices, that is, the matrices that commute with $A$.

\begin{corollary} \label{cor:2}
In addition to the hypotheses in Theorem \ref{thm:1}, suppose that $f$ is differentiable at $x_0$ and that $L,M \in \mathcal{L}(x_0,f)$ are differentiable at $0$. Then there exists a linear map $G:\R^n \to \R^n$ contained in $Z(f'(x_0))$ such that $M = L \circ G$.
\end{corollary}

Note that if $f'(x_0)$ is a scalar multiple of the identity, then $G$ can be any linear map, but in general there are more restrictions on what $G$ can be.

In dimension two and the holomorphic case, the assumption that $f,L,M$ are complex differentiable implies that the linear map $G$ must in fact be complex linear, i.e. $G(z) = \lambda z$ for some $\lambda \in \C$. This recovers the usual relationship between linearizers.

Next, let  $h:\R^n\to\R^n$ be a quasiregular mapping which is automorphic with respect to a discrete group $\Gamma$ of isometries of $\R^n$. This means that $h\circ\gamma = h$ for any $\gamma \in \Gamma$ and that $\Gamma$ acts transitively on the fibres $h^{-1}(y)$, that is, if $x_1,x_2 \in h^{-1}(y)$ then there exists $\gamma \in \Gamma$ such that $\gamma(x_1) = x_2$.

\begin{theorem}[\cite{IwMart}, pp. 501-502]
Let $\Gamma$ be a discrete group such that $h: \mathbb{R}^n \rightarrow \mathbb{R}^n$ is automorphic with respect to $\Gamma$. If there is a similarity $A = {\lambda}\mathcal{O}$, where $\lambda\in \mathbb{R}$ and $\mathcal{O}$ is an orthogonal transformation, such that
\[ A{\Gamma}A^{-1} \subset \Gamma, \] then there is a unique solution $f$ to the Schr\"oder functional equation \[ f \circ h = h \circ A\] and $f$ is a uniformly quasiregular mapping of $h(\mathbb{R}^n).$ 
\end{theorem}

The Schr\"oder equation is exactly the same functional equation satisfied by Poincar\'e linearizers, and hence the $\Gamma$-automorphic functions $h$ of the Schr\"oder equation are linearizers. It turns out that all the other linearizers in the same class as $h$ are also automorphic with respect to a quasiconformal group arising as a conjugate of $\Gamma$. Recall that a quasiconformal group consists of mappings which are all $K$-quasiconformal for a fixed $K$.

\begin{corollary}\label{cor:3}
Suppose $h: {\R}^n \rightarrow {\R}^n$ is a quasiregular mapping that is automorphic with respect to a discrete group $\Gamma$, and $f$ is a uniformly quasiregular mapping that is a solution to the Schr\"oder functional equation $f \circ h = h \circ A$, where $A=\lambda \mathcal{O}$ and $\lambda >1$. Then $h(0)$ is a repelling fixed point of $f$. If $L\in \mathcal{L}(h(0),f)$, then $L$ is a quasiregular mapping which is automorphic with respect to a quasiconformal group $\Gamma_G$, obtained by conjugating $\Gamma$ by a quasiconformal mapping $G$.
\end{corollary}

It is not hard to see that $e^z$ is both an automorphic holomorphic function with respect to the group of translations $\{ z\mapsto z+2\pi i k, k\in \Z \}$ and a linearizer for $z^2$ at $z=1$. Corollary \ref{cor:3} implies that the quasiregular version of this statement holds too: the Zorich map $Z:\R^3 \to \R^3$ (we restrict to dimension three for clarity) is both automorphic and a linearizer for a uqr power mapping. We give details for these observations in the final section.

\section{Relationship between linearizers}

\begin{proof}[Proof of Theorem \ref{thm:1}]
Let $f,x_0$ be as in the hypotheses and suppose that $L,M\in\mathcal{L}(x_0,f)$.
Let $\varphi, \psi \in Df(x_0)$ be the generalized derivatives associated with $L$ and $M$ respectively.  Then 
\begin{equation}\label{eq:4.1}
 L \circ \varphi  = f \circ L \text{ and } M \circ \psi =  f \circ M.
\end{equation}
Now, the proof of \cite[Theorem 6.3]{H.M.M.} shows that $L,M$ are injective in neighborhoods of $0$ $U,V$ respectively.
Since $L$ and $M$ are linearizers for $f$ at the repelling fixed point $x_0,$ then $L(U) \cap M(V) = E \ni x_0,$ which is a subset of the domain of $f$ for which the linearization is valid.  It follows that there exists a map $G =  L^{-1} \circ M$ that is quasiconformal in the domain $D = M^{-1}(E)$ with $0 \in D$.  

With \eqref{eq:4.1}, we can further extend the domain on which $G$ is defined by repeated application of the given functional equations:  
\begin{equation}\label{eq:4.2}
 L \circ \varphi^{k} = f^k \circ L \text{ and } M \circ \psi^k  = f^k \circ M.  
\end{equation}
First note that since $\psi$ and $\phi$ are loxodromic repelling uniformly quasiconformal mappings, if $U$ is any neighbourhood of $0$ then, given a compact set $K \subset \R^n$, there exists $M\in \N$ such that $K \subset \bigcup_{m=1}^M \psi^m(U)$. Further, $\R^n = \bigcup_{m\geq 1} \psi^m(U)$.
The corresponding properties holding for $\phi$ too.
Then using \eqref{eq:4.2}, for any $x \in \psi^k(D)$, we define $G: \psi^k(D) \rightarrow {\R}^n$ by 
\begin{equation*} G(x) = (\varphi^k \circ G \circ (\psi^{-1})^k)(x).
\end{equation*}
Since $\psi$ and $\varphi$ are uniformly quasiconformal, then $G$ is $K$-quasiconformal on $\psi^k(D)$.  

The above can be summarized by the following diagram:
\begin{center}
\begin{tikzpicture}

	\node (U) {$U$};
	\node (PU) [node distance=2.4cm, right of=U] {$\varphi^k (U)$};
	\node (B) [node distance=2.4cm, below of=U] {$E$};
	\node (V) [node distance=2.4cm, below of=B] {$V$};
	\node (PV) [node distance=2.4cm, right of=V] {$\psi^k (V)$};
	\node (FB) [node distance=2.4cm, right of=B] {$f^k(E)$};
	\draw[->] (U) to node [above] {$\varphi^k$} (PU);
	\draw[->] (U) to node [left] {$L$} (B);
	\draw[->] (B) to node [above] {$f^k$} (FB);
	\draw[->] (V) to node [left] {$M$} (B);
	\draw[->] (PU) to node [right] {$L$} (FB);
	\draw[->] (PV) to node [right] {$M$} (FB);
	\draw[->] (V) to node [below] {$\psi^k$} (PV);
	\draw[->, bend left] (V) to node [left] {$G$} (U);
	\draw[->, bend right] (PV) to node [right] {$G$} (PU);
         
\end{tikzpicture}
\end{center}
Since $\psi$ and $\varphi$ are uniformly quasiconformal, $G$ is $K$-quasiconformal on ${\R}^n$.  Also, since $G = L^{-1} \circ M$ on $D$, then $M = L \circ G$ on $D$.  But, since we can globally extend $G, L,$ and $M$, then $M = L \circ G$ on ${\R}^n$.  
\end{proof}

\begin{proof}[Proof of Corollary \ref{cor:0}]
Suppose that $L,M \in \mathcal{L}_1(x_0,f)$ and arise as limits of 
\[ L_j(x) =  f^{k_{u_j}}(x_0+xp_{k_{u_j}}) ,\quad M_j(x) =  f^{k_{v_j}}(x_0+xq_{k_{v_j}}).\]
Since $u_j = v_j$ for infinitely many $j$, we may relabel and assume that $L,M$ arise as limits of
\[ L_j(x) =  f^{k_j}(x_0+xp_{k_{j}}) ,\quad M_j(x) = f^{k_j}(x_0+xq_{k_{j}}).\]
Then $L_j^{-1}(M_j(x)) = q_{k_j}x/p_{k_j}$ in a neighborhood of $0$. By equation (6), p.86 and the proof of Theorem 6.3(ii) in \cite{H.M.M.}, there exists $C \geq 1$ such that
\[  |L_j^{-1}(M_j(x)) | \leq C, \quad |x| = 1,\]
for all $j$. Hence we may find a convergent subsequence of $q_{k_j}/p_{k_j}$ and again relabelling, we conclude that there exists $r>0$ such that $L_j^{-1}(M_j(x)) \to rx$ in a neighborhood of $0$. That is, $M(x) = L(rx)$ in a neighborhood of $0$, which is then extended to all of $\R^n$ as in the proof of Theorem \ref{thm:1}.

\end{proof}
 
\begin{proof}[Proof of Corollary \ref{cor:1}]
Suppose that $Df(x_0) = \{ \psi \}$.
Exactly as in the proof of Theorem \ref{thm:1}, we construct a quasiconformal map in a neighbourhood  of $0$ by $G=L^{-1}\circ M$ and then extend via the equation $G(x) = \psi^k(G(\psi^{-k}(x)))$. It is then clear that $G$ must commute with $\psi$.
\end{proof} 

\section{The Differentiable Case}

We recall some linear algebra; see \cite{Tropp} for some of the ideas used in this section on the Jordan canonical form. Denote by $||A||$ the operator norm of a matrix $A$, and by $||A||_{\infty}$ the norm
\[ ||A||_{\infty} =  \max_i \sum_{j=1}^n |A_{ij}| .\]
Since every two norms are equivalent on a finite dimensional space, there exists $C\geq 1$ such that
\begin{equation}\label{eq:norm}
\frac{||A||_{\infty}}{C} \leq ||A|| \leq C||A||_{\infty},
\end{equation}
for every $n\times n$ matrix $A$.

\begin{proof}[Proof of Lemma \ref{lem:uqr}]
Let $\psi:\R^n \to \R^n$ be linear and uniformly quasiconformal, and let $A$ be the $n\times n$ matrix representing $\psi$. 
Then $A = P^{-1}BP$ for some matrix $B$ in Jordan canonical form. That is
\[ B = \begin{pmatrix} \mathcal{J}_1 & & \\ &  \ddots & \\  & & \mathcal{J}_k \end{pmatrix},\quad \text{ where } \quad
\mathcal{J}_i = \begin{pmatrix} \lambda_i & 1 &  &  & \\ & \lambda _i & 1 &  & \\ & & \ddots & \ddots  & \\  & & & \lambda_i &1 \\
& &  & & \lambda_i \end{pmatrix}, \]
and the size of each Jordan block in $n_i$, we have $n_1 + \ldots + n_k = n$ and eigenvalues may be repeated.

We will show that under either the assumption that the eignevalues do not all have the same modulus, or under the assumption that $B$ is not diagonal, that
\[  \frac{ ||B^m||^n }{J_{B^m}} \to \infty\]
as $m\to \infty$. In other words, we will show that $B$ is not uniformly quasiregular. This then implies that $A$ is not uniformly quasiregular, since
\[ \frac{ ||A^m||^n}{J_{A^m}} = \frac{ ||P^{-1}B^mP||^n}{J_{P^{-1}}J_{B^m}J_P} = \frac{ ||P||^n\cdot ||P^{-1}B^mP||^n \cdot ||P^{-1}||^n }{||P||^n \cdot ||P^{-1}||^n J_{B^m} } \geq \frac{ ||B^m||^n}{||P||^n\cdot ||P^{-1}||^n J_{B^m} },\]
and $||P||,||P^{-1}||$ are non-zero.

First, suppose that the eigenvalues do not all have the same absolute value. Then there exists a largest, say $|\lambda_s|$, and some $\lambda_i$ with $|\lambda_i| < |\lambda_s|$. By the spectral radius formula $||B^m|| \sim |\lambda_s|^m$ for large $m$.
However, this implies
\[ \frac{ ||B^m||^n }{J_{B^m}} \sim \frac{ |\lambda_s | ^{mn}}{[\prod_{i=1}^k |\lambda_i|^{n_i} ]^m } \to \infty \]
as $m\to \infty$ since $|\lambda _s|^n > \prod_{i=1}^k |\lambda_i|^{n_i} $. This means that $B$ cannot be uniformly quasiregular.

Next, assume that $B$ is not diagonal, that is, there is at least one Jordan block of size at least $2$. Now, raising $B$ to a power yields
\[ B^m = \begin{pmatrix} \mathcal{J}_1^m & & \\ &  \ddots & \\  & & \mathcal{J}_k^m \end{pmatrix},\quad \text{ where } \quad
\mathcal{J}_i^m = \begin{pmatrix} \lambda_i^m & {m\choose 1} \lambda_i^{m-1} & {m\choose 2} \lambda_i^{m-2}  & \cdots & {m\choose n_i-1} \lambda_i^{m-n_i+1} \\ 
& \lambda _i^m & {m\choose 1} \lambda_i^{m-1} & \ldots & \vdots \\ & & \ddots & \ddots  & \vdots  \\  & & & \lambda_i^m &{m\choose 1} \lambda_i^{m-1} \\
& &  & & \lambda_i^m \end{pmatrix}, \]
and so we obtain
\[ ||B^m||_{\infty}= \max_{i} ||\mathcal{J}_i^m||_{\infty}.\]
Denote by $I$ the index that gives the maximum. 
Clearly we have
\begin{align*} 
||\mathcal{J}_I^m||_{\infty} &= |\lambda_I|^m + {m\choose 1} |\lambda_I|^{m-1} + \ldots + {m\choose n_I-1} |\lambda_I|^{m-n_I+1} \\
&= |\lambda_I|^m \left ( 1 + \frac{{m\choose 1}}{|\lambda_I|} + \ldots + \frac{ {m\choose n_I-1}}{|\lambda_I|^{n_I-1} } \right ).
\end{align*}
Therefore, by \eqref{eq:norm} and the fact that all eigenvalues have the same absolute value,
\begin{align*}
 \frac{ ||B^m||^n }{J_{B^m}} &\geq \frac{ ||B^m||^n_{\infty}}{C^nJ_{B^m}} \\
&= \frac{ || \mathcal{J}_I^m||^n_{\infty} }{C^nJ_{B^m} }\\
&= \frac{ |\lambda_I|^{mn}  \left ( 1 + \frac{{m\choose 1}}{|\lambda_I|} + \ldots + \frac{ {m\choose n_I-1}}{|\lambda_I|^{n_I-1} } \right )^n }{ C^n [\prod_{i=1}^k |\lambda_i|^{n_i}  ]^m} \\
&= \frac{   \left ( 1 + \frac{{m\choose 1}}{|\lambda_I|} + \ldots + \frac{ {m\choose n_I-1}}{|\lambda_I|^{n_I-1} } \right )^n }{ C^n }
\end{align*}
As $m\to \infty$, this latter expression diverges to infinity, and so again $B$ cannot be uniformly quasiregular. This completes the proof of the lemma.
\end{proof}

\begin{proof}[Proof of Corollary \ref{cor:2}]
By the hypotheses, $Df(x_0)$ is simple and contains only $\psi = f'(x_0)$. Suppose that $L,M \in \mathcal{L}(x_0,f)$.
By Corollary \ref{cor:2}, there exists a quasiconformal map $G:\R^n\to\R^n$ such that $M=L\circ G$ and $G$ commutes with $\psi(x) = f'(x_0)x$.

Since $L,M$ are differentiable at $0$, $G$ is differentiable at $0$. Then since $G(\psi(x)) = \psi(G(x))$, we have
\[ G'(0) \circ \psi'(0) = G'(\psi(0)) \circ \psi'(0) = (G\circ \psi)'(0) = (\psi\circ G)'(0) = \psi(G(0)) \circ G'(0) = \psi'(0) \circ G'(0).\]
That is, $G'(0)$ commutes with $\psi'(0)$, but $\psi'(0)$ is just the matrix $f'(x_0)$. Hence $G'(0) \in Z(f'(x_0))$.

We claim that $G(x) = G'(0)x$ for all $x\in \R^n$. Since $G$ is differentiable at $0$ and fixes $0$, we can write
\[ G(x) = G'(0)x + E(x), \text{ where } \lim_{x\to 0} \frac{|E(x)|}{|x|} = 0.\]
Now, given any $x\in \R^n$,
\begin{align*} 
G(x) &= \psi ^k(G(\psi^{-k}(x))) \\
&= \psi^k (G'(0)\psi^{-k}(x) + E(\psi^{-k}(x))) \\
&= \psi^kG'(0)\psi^{-k}(x) + \psi^k[E\psi^{-k}(x)],
\end{align*}
using the fact that $\psi$ is linear. Since $G'(0)$ commutes with $\psi$, the first term here is just $G'(0)x$. For the second term,
given $\epsilon >0$, there exists $\delta>0$ such that if $|x|<\delta$, then $|E(x)| < \epsilon |x|$. 
Here, we use $|x|$ to denote the Euclidean length of a vector in $\R^n$.

By Lemma \ref{lem:uqr}, we can write $\psi = P^{-1}BP$, where $B$ is a composition of a rotation and a dilation $x\mapsto \lambda x$ for some $\lambda >0$.
Hence given any $x\in \R^n$, there exists $N\in \N$ such that if $k\geq N$, we have
\[ | \psi^kE\psi^{-k}(x) | \leq  |P^{-1}B^kPEP^{-1}B^{-k}P(x)| \leq ||P||^2||P^{-1}||^2\lambda^k\lambda^{-k} \epsilon |x|
= ||P||^2||P^{-1}||^2\epsilon |x|.\]
Since this holds for every $\epsilon >0$, it follows that $ | \psi^kE\psi^{-k}(x) | = 0$ and hence
$G(x) = G'(0)x$ as claimed. 
\end{proof}

\section{Schr\"oder's Functional Equation}

\begin{proof}[Proof of Corollary \ref{cor:3}]
Since $f$ is uqr with repelling fixed point $h(0)$, there exists $L\in\mathcal{L}(h(0),f)$ with a generalized derivative $\psi$. Hence
$f \circ h = h \circ A$ and $f \circ L = L \circ \psi$. Using the same idea as in the proof of Theorem \ref{thm:1}, there exists a quasiconformal map $G:\R^n \to \R^n$ such that $L\circ G = h$. Note that $A$ is a similarity and so it is obviously uniformly quasiconformal.

Since $h$ is automorphic with respect to $\Gamma$, we have $h(\gamma(x)) = h(x)$ for all $\gamma \in \Gamma$ and $x\in \R^n$. Hence $L(G(\gamma(x))) = L(G(x))$,
which implies that
\[ L(\gamma_G(x)) = L(x),\]
for all $x\in \R^n$ and where $\gamma_G = G\circ \gamma \circ G^{-1}$. Hence $L$ is automorphic with respect to the quasiconformal group $\Gamma_G = \{ \gamma_G : \gamma \in \Gamma \}$.

\end{proof}

\section{The Zorich Map and the Latt\'{e}s-type Power Map}

The Zorich mapping, first constucted in \cite{Zorich}, is a quasiregular mapping $Z: {\R}^3 \rightarrow {\R}^3{\backslash}\{0\}$ that is analogous to the exponential map in the plane.  We recall the definition of $Z$ as follows.

Denote by $(u,v,w)$ coordinates on $\R^3$.
Partition the plane $\{w=0 \}$ into squares of length $\pi$, such that one square is centred at the origin.  Choose a bi-Lipschitz map
\[ h: \left [-\frac{\pi}{2}, \frac{\pi}{2} \right ]^2 \rightarrow \{(u,v,w): u^2 + v^2 + w^2 = 1, w \geq 0\}. \]  Here, as in \cite{FN}, we will consider 
\[ h(u,v) = \left (\frac{u\sin M(u,v)}{\sqrt{u^2 + v^2}}, \frac{v\sin M(u,v)}{\sqrt{u^2 + v^2}}, \cos (M(u,v)) \right ),\]  where $M(u,v) = $max$\{|u|,|v|\}$.  Then $h$ maps the square centered at zero of length $\pi$ onto the upper hemisphere of the unit sphere. We extend $h$ to the beam $[-\pi/2,\pi/2]^2 \times\R$ by $Z(u,v,w)=e^wh(u,v)$. The interior of this beam is mapped onto the upper half-space and the boundary is mapped onto the plane $\{w=0\}$ with the origin removed. 
We then extend $Z$ to the whole of $\R^3$ by repeated reflections in the sides of beams in the domain and in the plane $\{w=0\}$ in the image.

The resulting mapping is quasiregular and automorphic with respect to the group of translations 
\[ \Gamma = \{ \gamma(u,v,w) = (u+m\pi,v+n\pi,w): m,n \in \Z \} .\]
Following \cite{Mayer1}, there exists a uqr version of the power mapping which arises as a solution to the Schr\"oder equation
$P\circ Z = Z \circ d$, where $d(x) = m^2x$ for some $m\in \N$ with $m\geq 2$. Since $Z(0)=(0,0,1)$, $P$ has a repelling fixed point at $(0,0,1)$ and we can consider the set of linearizers $\mathcal{L}(Z(0),P)$.

Clearly $Z\in \mathcal{L}(Z(0),P)$, and for any other $L\in \mathcal{L}(Z(0),P)$, by Corollary \ref{cor:3}, $L$ must be automorphic with respect to a quasiconformal group arising as a quasiconformal conjugate of $\Gamma$.

Next, we will show that $Z\notin \mathcal{L}_1(Z(0),P)$. Let $x_0 = Z(0)=(0,0,1)$. Any $L\in \mathcal{L}_1(x_0,P)$ is the limit of a subsequence of $L_k(x) = P^k(x_0 + \lambda_kx)$, where $\lambda_k > 0$ for all $k$ and $\lambda_k \rightarrow 0$ as $k \rightarrow \infty$.  
Suppose that $Z\in \mathcal{L}_1(x_0,P)$. Then using the Schr\"oder equation and the fact that $Z^{-1}$ exists in a neighborhood of $x_0$, there exists a convergent subsequence of 
\[ L_k(x) = Z \circ d^k \circ Z^{-1} (x_0 + \lambda_kx),\]
in a neighborhood of $0$. Now, it is not hard to see that the infinitesimal space (recall section \ref{sec:t} and \cite{G.M.R.V.}) $T(x_0,Z^{-1})$ is simple and contains only the mapping
\[ g(u,v,w) = C(s(u,v),w/e), \]
where $s:\R^2 \to \R^2$ is the radial stretch which maps the circle of radius $r$ onto the square $\{ (u,v): \max \{|u|,|v|\} = r \}$ for $r>0$, and $C$ is a constant chosen so that $|g(B(0,1))| = |B(0,1)|$.
It is not hard to see that $g$ is homogeneous with $D=1$, recalling $\eqref{eq:hom}$.
We therefore have that
\begin{equation}
\label{eq:final} 
d^k(Z^{-1}(x_0 + \lambda_kx)) \approx d^k r(x_0,Z^{-1},\lambda_k)g(x) .
\end{equation}
The coefficient in front of $g$ here is just a real number for each $k$, and so by the definition of $g$ (in particular the shape of $g(S)$, where $S$ denotes the unit sphere in $\R^3$) there is no way that this expression can converge to the identity along any subsequence. In turn, this means that no subsequence of $L_k$ can converge to $Z$ and the claim follows.

We finally observe that by choosing the sequence $\lambda_k$ appropriately in \eqref{eq:final} and using the fact that $g$ is homogeneous, we have that 
\[ \mathcal{L}_1(Z(0),P) = \{ Z(g(rx)) : r>0 \}.\]

\end{document}